\documentclass[11pt,a4paper,dvips,a4wide]{article}
\usepackage{amsmath}
\usepackage{latexsym,epsfig,amssymb,amsthm,color,url,bbm,psfrag}

\newcommand{\pr}{\rightarrow}

\newcommand{\ba}{\begin{array}}
\newcommand{\ea}{\end{array}}

\newcommand{\eps}{\varepsilon}

\newcommand{\il}{\int\limits}
\newenvironment{inspring}[1]%
{\begin{list}{}{\setlength{\rightmargin}{0cm}
                \setlength{\listparindent}{0cm}
                \settowidth{\labelwidth}{\mbox{#1}}
                \setlength{\leftmargin}{1.1\labelwidth}
                \setlength{\labelsep}{.1\labelwidth}}}%
{\end{list}}

\newcommand{\bi}[1]{\begin{inspring}{#1}}
\newcommand{\ei}{\end{inspring}}

\newcommand{\dsum}{\displaystyle \sum}

\newcommand{\beq}{\begin{equation}}
\newcommand{\eq}{\end{equation}}
\catcode`\@=11

\font\tenmsa=msam10 \font\sevenmsa=msam7 \font\fivemsa=msam5
\font\tenmsb=msbm10 \font\sevenmsb=msbm7 \font\fivemsb=msbm5
\newfam\msafam
\newfam\msbfam
\textfont\msafam=\tenmsa  \scriptfont\msafam=\sevenmsa
  \scriptscriptfont\msafam=\fivemsa
\textfont\msbfam=\tenmsb  \scriptfont\msbfam=\sevenmsb
  \scriptscriptfont\msbfam=\fivemsb

\def\Bbb{\ifmmode\let\next\Bbb@\else
 \def\next{\errmessage{Use \string\Bbb\space only in math mode}}\fi\next}
\def\Bbb@#1{{\Bbb@@{#1}}}
\def\Bbb@@#1{\fam\msbfam#1}
\newcommand{\dR}{{\Bbb R}}

\newcommand{\dC}{{\Bbb C}}

\newtheorem{thm}{Theorem}
\newtheorem{lem}[thm]{Lemma}
\newtheorem{prop}[thm]{Note}
\numberwithin{thm}{section}
\newcommand{\e}{{\rm e}}
\newcommand{\eqdist}{\,{\buildrel d \over =}\,
}
\usepackage{fullpage}

\begin{document}

\title{{\Large {\bf Asymptotics of the maximum of Brownian motion\\
under Erlangian sampling\\ \vspace{.5cm}
{\large \rm paper in honor of N.G.\ de Bruijn}\\
\vspace{.5cm}
}}}

\author{
A.J.E.M. Janssen\footnotemark[1] \and
J.S.H. van Leeuwaarden\footnotemark[2]
        }

\date{\today}

\maketitle

\footnotetext[1]{Eindhoven University of Technology and Eurandom, Department of Mathematics and Computer Science and Department of Electrical Engineering, P.O. Box 513, 5600 MB Eindhoven, The Netherlands. E-mail a.j.e.m.janssen@tue.nl
}

\footnotetext[2]{Eindhoven University of Technology, Department of Mathematics and Computer Science, P.O. Box 513, 5600 MB Eindhoven, The Netherlands. E-mail j.s.h.v.leeuwaarden@tue.nl
}

\begin{abstract}
Consider the all-time maximum of a Brownian motion with negative drift. Assume that this process is sampled at certain points in time, where the time between two consecutive points is rendered by an Erlang distribution with mean $1/\omega$. The family of Erlang distributions covers the range between deterministic and exponential distributions. We show that the average convergence rate as $\omega\to\infty$ for all such Erlangian sampled Brownian motions is $O(\omega^{-1/2})$, and that the constant involved in $O$ ranges from $-\zeta(1/2)/\sqrt{2\pi}$ for deterministic sampling to $1/\sqrt{2}$ for exponential sampling. The basic ingredients of our analysis are a finite-series expression for the expected maximum, an asymptotic expansion of $\sum_{j=1}^{k-1}(1-\exp(2\pi i j/k))^{-s}$, $s\in\mathbb{R}$, as $k\to\infty$ using Euler-Maclaurin summation, and Fourier sampling of functions analytic in an open set containing the closed unit disk.

\vspace{1.5mm}


\vspace{1.5mm}


\end{abstract}

\section{Introduction}
Let $\{B_\beta(t):t\geq 0\}$ be a Brownian motion with negative drift whose position at time $t$ is given by
\begin{equation}
B_\beta(t)=-\beta t + W(t),  \quad \beta\geq 0,
\end{equation}
with $B_\beta(0) = 0$ and $\{W(t):t\geq 0\}$  a Wiener process (standard Brownian motion). Since $\beta$ is assumed to be positive, the Brownian motion will eventually drift towards $-\infty$, and the all-time maximum $\tilde M_\beta=\sup_{t\in\mathbb{R}^+}B_\beta(t)$ is well defined. In fact, it is  known that $\tilde M_\beta$ follows an exponential distribution with rate  $2\beta$, so that
$\mathbb{P}(\tilde M_\beta\geq x)=e^{-2\beta x}$ (see e.g.~\cite[Lemma 5.5]{chenyao}), and hence the expected all-time maximum is simply given by $\mathbb{E}\tilde M_\beta=1/2\beta$.

We   consider sampled versions of the Brownian motion, meaning that  we observe the process only at time points $t_0=0, t_1, t_2,\ldots$. A crucial assumption we make is that the times between consecutive sampling points
$T_n=t_n-t_{n-1}$, $n\in\mathbb{N}$, are independent and identically distributed (i.i.d.). Let $T$ denote a generic random variable with $T\eqdist T_1$ (here $\eqdist$ denotes `equal in distribution') and $\mathbb{E}T=\omega^{-1}$. The constant $\omega$ represents the expected number of observations per unit of time, henceforth referred to as the {\it sampling frequency}.

It is readily seen that a sampled version of the Brownian motion constitutes a random walk $\{S_\beta(n):n\in\mathbb{N}\}$ with $S_\beta(0)=0$ and
\begin{equation}
S_\beta(n)=X_1+\ldots+X_n \quad {\rm with} \quad  X_1,X_2,\ldots \ {\rm i.i.d.}, \ X_1\eqdist B_\beta(T).
\end{equation}
The fact that Brownian motion evolves in continuous space and time leads to great simplifications in determining its
properties. In contrast, the random walks that we obtain after sampling, moving only at certain points in time,
are objects that are much harder to study. Although it seems plausible that, as $\omega\to\infty$, the behavior of the
random walk should be similar to that of Brownian motion, there
are many effects to take into account for finite $\omega$. Let the maximum of the random walk be denoted by $
{M}_\beta(\omega)=\sup_{n=0,1,\ldots}B_\beta(t_n)$. The sampling error $\Delta_\beta(\omega)=\tilde M_\beta-{M}_\beta(\omega)$ then depends on the drift $\beta$, the sampling frequency $\omega$, and of course on the distribution of $T$.
This paper deals with the expected maximum
of the random walks and, in particular, its deviation $\mathbb{E}\Delta_\beta(\omega)$ from the expected maximum $1/2\beta$ of the underlying Brownian motion. This relatively simple characteristic already turns out to
have an intriguing description.

We   assume that the times between sampling points are drawn from an Erlang distribution, so that $T\eqdist E_k(\lambda)$
with $E_k(\lambda)$ an Erlang distributed random variable consisting of $k$ independent exponential phases, each with mean $1/\lambda$, and
\begin{equation}
\mathbb{P}(E_k(\lambda)<x)=1-\sum_{n=0}^{k-1}\frac{1}{n!}\e^{-\lambda x}(\lambda x)^n, \quad x\geq 0.
\end{equation}
The random variable $E_k(\lambda)$ has mean $k/\lambda$ and  variance $k/\lambda^2$. One reason for working under the assumption of Erlangian sampling is that $T$ constant  and $T$ exponentially distributed are opposite extremes with regard to randomness as well as in the family of Erlang distributions (viz.~with $\lambda=k\omega$ and $k\to\infty$ and $k=1$, respectively). Another reason is that Erlangian sampling leads to a random walk of which the distribution of the all-time maximum allows for an explicit solution. This gives rise to a series expression for $\mathbb{E}M_\beta(\omega)=\mathbb{E}M_\beta(\omega;k)$ in which the $k$ terms involve the $k$ roots of $P(\sigma)=[\sigma(1+\rho-\sigma)]^k-\rho^k=0$ in $|\sigma|<1$ with $\rho\in(0,1)$ given by $k(1-\rho)^2\omega=2\beta^2\rho$. In this paper this series expression is analyzed for the case that $\omega\to\infty$, and a major result is that
\begin{equation}\label{rrs}
\mathbb{E}M_\beta(\omega;k)=\frac{1}{2\beta}-\frac{\varphi_k}{\sqrt{2\pi\omega}}+O(\omega^{-1})
\end{equation}
where $O(\omega^{-1})$ holds uniformly in $k\geq 1$ as $\omega\to\infty$, and
$\varphi_k\to\zeta(1/2)$ as $k\to\infty$. More than fifty years after its appearance, N.G.~de Bruijn's {\it Asymptotic Methods in Analysis} \cite{debruijn}, in particular Sec.~1.2 on the $O$-symbol, and Secs.~3.6-10 on the Euler-Maclaurin summation, still provides us guidance in doing the asymptotic analysis as required for establishing a result like \eqref{rrs}. 
For a more detailed overview of our results we refer to Subsection \ref{subover}. Other recent works that exploit the beneficial properties of Erlangian sampling are \cite{ref111} for compound Poisson processes and \cite{ref2} for L\'{e}vy processes.

\section{Preliminaries and overview}
In this section we present some preliminary results in Subsection \ref{prel} and an overview of the main results in Subsection \ref{subover}. Subsection \ref{prel} starts with results for the special cases of equidistant and exponential sampling, and then we derive
a general expression for the expected all-time maximum for Erlangian sampling.

\subsection{Special cases of Erlangian sampling}\label{prel} In the case of equidistant sampling the time between two consecutive sampling points is always $1/\omega$. From the definition of Brownian motion it then immediately follows that
$$B_\beta(T)\eqdist N(-\beta\omega^{-1},\omega^{-1}),$$ where $N(a,b)$ denotes a normally distributed random variable with mean $a$ and variance $b$.
 We should thus consider the maximum of a random walk with normally distributed increments, referred to in the literature as the Gaussian random walk. The maximum of this random walk was studied in  \cite{changperes,jllerch}. In particular, \cite[Thm.~2]{jllerch} yields, for $\omega>\beta/(2\sqrt{\pi})$,
\begin{eqnarray}\label{wdfegfw571}
\mathbb{E}{M}_\beta(\omega)= \frac{1}{2\beta}+\frac{\zeta(1/2)}{\sqrt{2\pi \omega}}+\frac{\beta}{4\omega}+\frac{\beta^2}{\omega\sqrt{2\pi\omega}}\sum_{r=0}^{\infty}\frac{\zeta(-1/2-r)}{r!(2r+1)(2r+2)}\left(\frac{-\beta^2}{2 \omega}\right)^r.
\end{eqnarray}
 This implies immediately that
\begin{equation}\label{deltadet}
\mathbb{E}\Delta_\beta(\omega)=
-\frac{\zeta(1/2)}{\sqrt{2\pi \omega}}+O (\omega^{-1})
\end{equation}
with $-\zeta(1/2)/\sqrt{2\pi}\approx 0.5826$. Results similar to (\ref{deltadet}), in slightly different settings, have been presented in \cite[Thm.~2]{asmussenglynnpitman} and \cite[Thm.~1]{calvin}. A crucial difference is that our result (\ref{wdfegfw571}) is obtained from the exact expression for $\mathbb{E}{M}_\beta(\omega)$, while the results in \cite{asmussenglynnpitman,calvin} are derived from considering the Brownian motion in a finite time interval, and estimating its maximum by Euler-Maclaurin summation.


In the case of exponential sampling, we assume that the times between consecutive sampling points are independent and exponentially distributed with mean $1/\omega$. In this case we can prove that (see Lemma \ref{decmm2} with $k=1$)
\begin{equation}
B_\beta(T)\eqdist E_1(\gamma_2)-E_1(\gamma_1),
\end{equation}
where
\begin{equation}\label{gammas}
\gamma_1=-\beta+\sqrt{\beta^2+2\omega}, \quad \gamma_2=\beta+\sqrt{\beta^2+2\omega}.
\end{equation}
The random walk for which the increments are distributed as the difference of two exponentials has been thoroughly studied in the literature.
The maximum of this random walk is known to be equal in distribution to the stationary waiting time in a so-called $M/M/1$ queue with arrival rate $\gamma_1$ and service rate $\gamma_2$, for which (see e.g.~\cite[p.~108]{asmussen})
\begin{equation}\label{}
\mathbb{P}(M_\beta(\omega)>x)=\frac{\gamma_1}{\gamma_2}e^{-(\gamma_2-\gamma_1)x}, \quad x\geq 0.
\end{equation}
This implies that
\begin{equation}\label{}
\mathbb{E}M_\beta(\omega)=\frac{\gamma_1}{\gamma_2}\frac{1}{\gamma_2-\gamma_1},
\end{equation}
from which it readily follows that
\begin{equation}\label{asex}
\mathbb{E}\Delta_\beta(\omega)=
\frac{1}{\sqrt{2 \omega}}+O (\omega^{-1})
\end{equation}
with $1/\sqrt{2}\approx 0.7071$. A similar result was obtained in \cite[Thm.~3]{calvin} for a Brownian motion in a finite time interval sampled at uniformly distributed points.

We next set $T\eqdist E_k(k\omega)$ with mean $\omega^{-1}$ and variance $(k\omega^2)^{-1}$. Notice that random sampling ($k=1$) and equidistant sampling $(k=\infty)$ can be seen as special cases. We first make the following observation.
\begin{lem}\label{decmm2}
If $T\eqdist E_k(k\omega)$ then
\begin{equation}
B_\beta(T)\eqdist E_k(\gamma_2 )-E_k(\gamma_1 ),
\end{equation}
where
\begin{equation}\label{e2}
\gamma_1 =-\beta+\sqrt{\beta^2+2k\omega}, \quad \gamma_2 =\beta+\sqrt{\beta^2+2k\omega}.
\end{equation}
\end{lem}
\begin{proof}
For $\beta s+s^2/2<k\omega$,
\begin{align}
\mathbb{E}(e^{-sB_\beta(E_1(k\omega))})&=\frac{k\omega}{k\omega-(\beta s+s^2/2)} \nonumber\\
&=\frac{\gamma_1 \gamma_2 }{(\gamma_1 -s)(\gamma_2 +s)}.
\end{align}
Hence,
\begin{equation}
\mathbb{E}(e^{-sB_\beta(T)})=\left(\frac{\gamma_1 }{\gamma_1 -s}\right)^k\left(\frac{\gamma_2 }{\gamma_2 +s}\right)^k,
\end{equation}
which completes the proof by L\'{e}vy's continuity theorem for Laplace transforms \cite{lev}.
\end{proof}

From Lemma \ref{decmm2} we conclude that in order to study an Erlangian sampled Brownian motion, we need to study a random walk with increments defined as the difference of two Erlang distributed random variables. As it happens, this random walk has been studied before, and an explicit solution for the distribution of $M_\beta(\omega)$ is available. In order to explain this, we need to make a small excursion into the world of fluctuation theory. We start from the observation that  $M_\beta(\omega)\eqdist W$ with
\begin{equation}\label{lind}
W\eqdist \max\{0,W+B_\beta(T)\}.
\end{equation}
This equation is a special case of {\it Lindley's equation}, describing the stationary waiting time of a customer in the $GI/G/1$ queue. The case in \eqref{lind} describes a single-server queue with Erlang distributed service times and Erlang distributed interarrival times. In \cite{adanzhao} it was shown that the distribution of $W$ can be expressed as a finite sum of exponentials, the exponents of which are the roots of an equation that involves the Laplace transform of the Erlang distribution. This gives the result \cite[Corollary 3.3]{adanzhao}
\begin{equation}\label{}
\mathbb{P}(M_\beta(\omega)>t)=\sum_{j=0}^{k-1} c_j\sigma_j e^{-\gamma_2(1-\sigma_j)t},  \quad t\geq 0.
\end{equation}
Here,
\beq \label{e6}
c_j=\frac{{\prod_{l=0,l\neq j}^{k-1}}\,(\sigma_l-1)}{{\prod_{l=0,l\neq j}^{k-1}}({\sigma_l} /{\sigma_j}-1)}~,\quad j=0,1,...,k-1,
\eq
and $\sigma_0,\ldots,\sigma_{k-1}$ the $k$ roots in $|\sigma|<1$ of
\beq \label{e3}
P(\sigma):=[\sigma(1+\rho-\sigma)]^k-\rho^k=0~,
\eq
where
\beq \label{e4}
\rho=1-\frac{2\beta}{\gamma_2}~,
\eq
so that $\frac{1}{\gamma_2}=\frac{1-\rho}{2\beta}$ and
\beq \label{e5}
k(1-\rho)^2\,\omega=2\beta^2\rho~.
\eq
The $\sigma_j$ are given explicitly as
\beq \label{e7}
\sigma_j=\tfrac12\,(1+\rho)-\sqrt{(\tfrac12\,(1+\rho))^2-\rho e^{2\pi ij/k}}~,~~~~~~j=0,1,...,k-1~.
\eq
The expected all-time maximum then equals
\begin{eqnarray} \label{e1}
\mathbb{E}M_\beta(\omega;k)=\sum_{j=0}^{k-1}\,\frac{c_j\sigma_j}{\gamma_2(1-\sigma_j)},
\end{eqnarray}
and it is this expression that forms the point of departure for this paper. Notice that $\sigma_{k-j}=\sigma_j^*$ for $j=1,\ldots,k-1$ and that $\sigma_0=\rho$ is real. Then it follows from \eqref{e6} that $c_{k-j}=c_j^*$ for $j=1,\ldots,k-1$ and that $c_0$ is real. Now \eqref{e1} implies that $\mathbb{E}M_\beta(\omega;k)$ is real.

\subsection{Overview of the results}\label{subover}
In the coordinates $\rho$, $k$, $\omega$ and $\beta$, related according to \eqref{e2}, \eqref{e4}, and \eqref{e5}, we obtain a limit result for
$\mathbb{E}M_\beta(\omega;k)$ in the case that $\omega\to \infty$ and $k$ bounded or unbounded, and in the case that $\omega\to 0$ and $k\to\infty$. In this paper, we do not address the intriguing question what happens when $\omega$ tends to a non-zero finite limit and $k\to\infty$.

In Section~\ref{sec2} we start from  \eqref{e1} and show that when $\omega\pr\infty$ and $k$ is bounded,
\beq \label{e8}
\mathbb{E}M_\beta(\omega;k)=\frac{1}{2\beta}-\frac{1}{\sqrt{2\omega k}}\,\Bigl( k-\sum_{j=1}^{k-1}\,\frac{1}{\sqrt{1-u_j}}\Bigr)+O(\omega^{-1})~
\eq
with
\beq \label{e9}
u_j=e^{2\pi ij/k}~,~~~~~j=0,1,...,k-1
\eq
the $k$ unit roots. Notice that the case $k=1$ is in line with \eqref{asex}.

In Section~\ref{sec3} we determine the asymptotic behavior of the series
\beq \label{e10}
\sum_{j=1}^{k-1}\,(1-u_j)^{-s}
\eq
when $s\in\mathbb{R}$ is fixed and $k\pr\infty$. We use for this a method developed by Brauchart et al.\ \cite{ref1} based on the Euler-Maclaurin summation formula. In particular, our results imply that
\beq \label{e11}
\frac{1}{\sqrt{2k}}\,\Bigl(k-\sum_{j=1}^{k-1}\,\frac{1}{\sqrt{1-u_j}}\Bigr)= \frac{-\zeta(1/2)}{\sqrt{2\pi}}- \frac{\pi^{1/2}\,\zeta(-1/2)}{2\sqrt{2}}~\frac1k+O(k^{-2})~.
\eq
Notice that for $k\to\infty$ the right-hand side of \eqref{e11} tends to the leading-order term at the right-hand side of  \eqref{deltadet}.

In Section~\ref{sec2} the result \eqref{e8} is proved for the case that $k$ remains bounded (Theorem~\ref{thm2.1}). The proof of Theorem~\ref{thm2.1} is relatively simple and uses the direct representation in \eqref{e6} of the $c_j$'s as they occur in the series expression \eqref{e1}. A key observation in this proof is that the term with $j=0$ in the series representation dominates all other terms. The result \eqref{e8} comprises the quantity $k-\sum_{j=1}^{k-1}(1-u_j)^{-1/2}$ for which the asymptotic behavior as $k\to\infty$ is given (Theorem~\ref{thm3.1}) in terms of the Riemann zeta function by \eqref{e11}, using an approach based on Euler-Maclaurin summation. It is therefore a relevant question to ask whether \eqref{e8} also holds with $O(\omega^{-1})$ holding uniformly in $k\geq 1$. The approach to prove \eqref{e8} for unbounded $k$ using the direct representation \eqref{e6} of the $c_j$ is severely complicated by the fact that, as $k\to\infty$ and $\omega\to\infty$, the $k$ zeros of $P$ in \eqref{e3} inside the unit circle accumulate on the set $\{1-\sqrt{1-u}\ | \ |u|=1\}$, so that the quantities $\sigma_l/\sigma_j-1$, $l\neq j$, that occur in \eqref{e6}, can become arbitrarily small. In a situation like this it may be advantageous, as exemplified on several occasions by N.G. de Bruijn in
\cite{debruijn}, to view the problem at hand from a different perspective. In this spirit, we use a different representation of the $c_j$'s, viz.~one in which factors $\sigma_l^{-}/\sigma_j^{+}-1$ (rather than $\sigma_l^{-}/\sigma_j^{-}-1$) appear with $\sigma^{-}$ and $\sigma^{+}$ the zeros of $P$ inside and outside the open unit disk $|\sigma|<1$, respectively (Lemma~\ref{lem4.1}). This new representation gives bounds on $c_j$ (Theorems~\ref{thm4.3} and \ref{thm4.4}) from which we can conclude that also in the case of unbounded $k$, the term with $j=0$ in the representation \eqref{e1} is dominant when $\omega\to\infty$. The proofs of
Theorems~\ref{thm4.3} and \ref{thm4.4} rely heavily on Lemma \ref{lem4.2} that gives bounds on the products $\prod_{l=0}^{k-1}(1-\sigma_l^{-}/\sigma_j^{+})$ that appear in the new representation of the $c_j$. In the proof of  Lemma \ref{lem4.2} a crucial role is played by the result that expresses a series $\sum_{j=0}^{k-1}h(u_j)$, with $h(z)=\sum_{n=0}^\infty d_n z^n$ analytic in an open set containing the closed unit disk, in terms of the decimated coefficients $d_{sk}$, $s=0,1,\ldots$ (Fourier sampling). In proving the extension of \eqref{e8} to unbounded $k$ (Theorem~\ref{thm5.1}), due to the results of Section \ref{sec4}, attention can thus be restricted to the term with $j=0$ in \eqref{e1}, making the proof rather straightforward.

In Section \ref{sec6} we consider the behavior of $\mathbb{E}M_\beta(\omega;k)$ as $\omega\to 0 $ and $k\to\infty$ (the case that $\omega\to 0$ while $k$ remains bounded yields $\mathbb{E}M_\beta(\omega;k)\to 0$ in a trivial manner from (\ref{e6})-(\ref{e1})).
Then the term with $j=0$ in \eqref{e1} is no longer dominant, and it can be shown from the results of Section~\ref{sec4} that $\mathbb{E}M_\beta(\omega;k)$ is well approximated, see the proof of Theorem~\ref{thm6.1}, by
\begin{eqnarray} \label{1212}
\frac{1-\rho}{2\beta k} \sum_{j=0}^{k-1} \frac{(1+\rho-\sigma_j)\, \sigma_j}{(1-\sigma_j)^2(1+\rho-2\sigma_j)}.
\end{eqnarray}
The series in \eqref{1212} can be cast into the form $\sum_{j=0}^{k-1}F(u_j)$, with $F$ analytic in an open set containing the closed unit disk and $F(0)=0$. To this series, the Fourier sampling technique, as it occurs in the proof of Lemma~\ref{lem4.2}, can be applied. It thus follows that $\mathbb{E}M_\beta(\omega;k)$ tends to zero when $\omega\to 0$ and $k\to\infty$, and $k,\omega$ related as in \eqref{e5} with fixed $\beta>0$, and also the rate at which $\mathbb{E}M_\beta(\omega;k)$ tends to zero can be determined.


\section{Behavior of $\mathbb{E}M_\beta(\omega;k)$ for bounded $k$ and $\omega\pr\infty$} \label{sec2}
We prove in this section the following result.
\begin{thm} \label{thm2.1}
Let $k=1,2,...\,$. As $\omega\pr\infty$,
\beq \label{e12}
\mathbb{E}M_\beta(\omega;k)=\frac{1}{2\beta}-\frac{1}{\sqrt{2\omega k}}\,\Bigl(k-\sum_{j=1}^{k-1}\, \frac{1}{\sqrt{1-u_j}}\Bigr)+O(\omega^{-1})~,
\eq
where $u_j$ are given in \eqref{e9}.
\end{thm}
\begin{proof} We use the series expression (\ref{e1}) for $\mathbb{E}M_\beta(\omega;k)$. We have from (\ref{e2}) and (\ref{e4}) that $1-\rho=O(\omega^{-1/2})$, and so from (\ref{e7})
\begin{eqnarray} \label{e13}
1-\sigma_j & = & \tfrac12\,(1-\rho)+\sqrt{1-u_j-(1-\rho)(1-u_j)+(\tfrac12\,(1-\rho))^2}~\nonumber \\
& = & \sqrt{1-u_j}+O(\omega^{-1/2})~,~~~~~~j=0,1,...,k-1~.
\end{eqnarray}
In particular, $1-\sigma_0=O(\omega^{-1/2})$ while $1-\sigma_j$ is bounded away from 0 for $j=1,...,k-1$ as $\omega\pr\infty$. Then, from (\ref{e6}) and $1-\rho=O(\omega^{-1/2})$, as $\omega\pr\infty$
\beq \label{e14}
c_0=\frac{{\prod_{l=1}^{k-1}}\,(\sigma_l-1)}{{\prod_{l=1}^{k-1}}\,\Bigl(\dfrac{\sigma_l}{\rho}-1\Bigr)} \pr1~,
\eq
while for $j=1,...,k-1$
\beq \label{e15}
\frac{c_j}{1-\rho}={-}\,\prod_{l=1,l\neq j}^{k-1}\,(\sigma_l-1)\:/\:\prod_{l=0,l\neq j}^{k-1}\, \Bigl(\frac{\sigma_l}{\sigma_j}-1\Bigr)
\eq
has a finite limit $\neq\:0$ as $\omega\pr\infty$. We conclude from $\rho=\sigma_0$, (\ref{e14}), (\ref{e15}) and $1-\rho=O(\omega^{-1/2})$ that, as $\omega\pr\infty$,
\beq \label{e16}
\mathbb{E}M_\beta(\omega;k)=\sum_{j=0}^{k-1}\,\frac{c_j\sigma_j}{\gamma_2(1-\sigma_j)}=\frac{1}{\gamma_2}~
\frac{c_0\rho}{1-\rho}+O(\omega^{-1})~.
\eq
To proceed, we need to approximate $c_0$ accurately. From the identity in the first line of (\ref{e13}) we find for $j=1,...,k-1$
\begin{eqnarray} \label{e17}
\sigma_j & = & 1-\tfrac12 (1-\rho)-\sqrt{1-u_j}\,(1-\tfrac12 (1-\rho)+O((1-\rho)^2))~\nonumber \\
& = & 1-\sqrt{1-u_j}+\tfrac12\,(\sqrt{1-u_j}-1)(1-\rho)+O((1-\rho)^2)~.
\end{eqnarray}
This gives, see (\ref{e14}),
\allowdisplaybreaks
\begin{eqnarray} \label{e18}
\rho\,c_0 & = & \rho^k\,\prod_{j=1}^{k-1}\,\frac{1-\sigma_j}{\rho-\sigma_j}~\nonumber \\
& = & \rho^k\,\prod_{j=1}^{k-1}\,\frac{\sqrt{1-u_j}-\frac12\,(\sqrt{1-u_j}-1)(1-\rho)+O((1-\rho)^2)} {\sqrt{1-u_j}-\frac12(\sqrt{1-u_j}+1)(1-\rho)+O((1-\rho)^2)}~\nonumber \\
& = & \rho^k\Bigl(1+\sum_{j=1}^{k-1}\,\frac{1-\rho}{\sqrt{1-u_j}}\Bigr)+O((1-\rho)^2)~\nonumber \\
& = & 1-(1-\rho)\Bigl(k-\sum_{j=1}^{k-1}\,\frac{1}{\sqrt{1-u_j}}\Bigr)+O((1-\rho)^2)~,
\end{eqnarray}
where it has been used that $\rho^k=1-(1-\rho)\,k+O((1-\rho)^2)$. Using this in (\ref{e16}) while noting that $\gamma_2(1-\rho)=2\beta$ and that $\frac{1}{\gamma_2}=\frac{1}{\sqrt{2\omega k}}+O(\omega^{-1})$, we obtain the result. \end{proof}

The proof of Theorem~\ref{thm2.1} shows that the term with $j=0$ in the series (\ref{e1}) dominates all other terms when $\omega\pr\infty$ and $k$ remains bounded. We establish a similar result more generally, allowing $k$ to be unbounded as well, in Section~\ref{sec5}, and extend the result of Theorem~\ref{thm2.1} accordingly.

\section{Asymptotics of $\sum_{j=1}^{k-1}\,(1-u_j)^{-s}$ as $k\pr\infty$ for fixed $s$} \label{sec3}

The large-$\omega$ expression in Theorem~\ref{thm2.1} for $\mathbb{E}M_\beta(\omega;k)$ contains the series $\sum_{j=1}^{k-1}\,(1-u_j)^{-s}$ with $s=1/2$. It is of interest to find out how this series behaves with increasing $k$. Furthermore, in Section~\ref{sec5}, we consider the case that $\omega\pr\infty$ with unbounded $k$ allowed, and then it appears that the behavior of the series $\sum_{j=1}^{k-1}\,(1-u_j)^{-s}$ for large $k$ is required to be known for $s=1/2,1,3/2,...\,$.

We adopt an approach in \cite{ref1}, for determining the asymptotic behavior of
\beq \label{e19}
S_k(s):=\sum_{j=1}^{k-1}\,(1-u_j)^{-s}~,~~~~~~u_j=e^{2\pi ij/k}~,
\eq
as $k\pr\infty$ and $s\in\dR$ is fixed. In \cite{ref1} this approach is used for finding the asymptotic behavior of
\beq \label{e20}
U_k(s):=\sum_{j=1}^{k-1}\,\Bigl(\sin\,\frac{\pi j}{k}\Bigr)^{-s}
\eq
as $k\pr\infty$ and $s\in\dC$ is fixed. The result that we obtain here for $S_k(s)$ is of the same nature as the result for $U_k(s)$ in \cite{ref1}, except that in our result, Theorem \ref{thm3.1} below, all terms
\beq \label{e21}
\zeta(s-l)\,k^{s-l}~,~~~~~~l=0,1,...~,
\eq
occur, while the result in \cite{ref1}\ has only terms (\ref{e21}) with even $l$. Furthermore, the exceptional cases $s=1,2,...$ are less complicated for our $S_k(s)$ than they are for $U_k(s)$ in \cite{ref1}.

The method in \cite{ref1}\ is a fine application of the Euler-Maclaurin summation formula that can be found, along with various applications, in \cite[Secs.~3.6--10]{debruijn}. We take $s\in\dR$ in (\ref{e19}) and this implies that the terms $t_j=(1-u_j)^{-s}$ in (\ref{e19}) satisfy $t_{k-j}=t_j^{\ast}$, $j=1,...,k-1$. Hence, $S_k(s)$ is real, and so we have
\beq \label{e22}
S_k(s)=\sum_{j=1}^{k-1}\,f_k(j)~,
\eq
where
\beq \label{e23}
f_k(x)={\rm Re}\,[(1-e^{2\pi ix/k})^{-s}]~.
\eq
Expanding
\beq \label{e24}
\Bigl(\frac{z}{e^z-1}\Bigr)^s=\sum_{l=0}^{\infty}\,\frac{B_l^{(s)}(0)}{l!}\,z^l~,~~~~~~|z|<2\pi~,
\eq
with $B_l^{(s)}$ the generalized Bernoulli polynomials, as in \cite[Sec.~1]{ref1}, we get
\beq \label{e25}
f_k(x)=\sum_{l=0}^{\infty}\,\beta_l(s)\Bigl(\frac{x}{k}\Bigr)^{l-s}~,~~~~~~|x|<k~,
\eq
where
\beq \label{e26}
\beta_l(s)=\frac{(2\pi)^{l-s}\,B_l^{(s)}(0)}{l!}\,\cos(l+s)\,\frac{\pi}{2}~,~~~~~~l=0,1,...~.
\eq
Remember that $\zeta$ denotes the Riemann zeta function.

\begin{thm} \label{thm3.1}
Let $s\in\dR$ and let $p=0,1,...$ such that $s+2p>0$. Then, as $k\pr\infty$,
\beq \label{e27}
S_k(s)=\left\{\ba{l}
k+2\,\dsum_{l=0}^{2p}\,\beta_l(s)\,\zeta(s-l)\,k^{s-l}+O_{s,p}(k^{s-2p-1})~, \quad s\neq 1,2,...~, \\
\tfrac12 k+2\,\dsum_{l=0,l\neq s-1}^{2p}\,\beta_l(s)\,\zeta(s-l)\,k^{s-l}+O_{s,p}(k^{s-2p-1})~, \quad s=1,2,...~,
\ea
\right.
\eq
where the constants implied by the $O_{s,p}$ depend on $s$ and $p$ but not on $k$.
\end{thm}
\begin{proof}
We closely follow the approach in \cite{ref1}, so that many of the details are left out. We have for $f\in C^{2p+1}([1,m])$ that
\begin{eqnarray} \label{e28}
& \mbox{} & \hspace*{-1cm}\sum_{j=1}^m\,f(j)=\il_1^m\,f(x)\,dx+\tfrac12\,(f(1)+f(m))~\nonumber \\
& & \hspace*{-1cm}+~\sum_{r=1}^p\,\frac{B_{2r}}{(2r)!}\,(f^{(2r-1)}(m)-f^{(2r-1)}(1))+\il_1^m\, f^{(2p+1)}(x)\, \frac{C_{2p+1}(x)} {(2p+1)!}\,dx~,
\end{eqnarray}
where $C_{2p+1}(x)=B_{2p+1}(x-[x])$ is the periodized Bernoulli polynomial and $B_{2r}$ is the Bernoulli number. Using this with $m=k-1$ and $f=f_k$, see (\ref{e23}), and noting that $f_k(k-x)=f_k(x)$, we get
\begin{eqnarray} \label{e29}
& \mbox{} & \hspace*{-7mm}S_k(s)=\sum_{j=1}^{k-1}\,f_k(j)~\nonumber \\
& & \hspace*{-7mm}=~2\,\il_1^{k/2}\,f_k(x)\,dx{+}f_k(1){-}2\,\sum_{r=1}^p\,\frac{B_{2r}}{(2r)!}\,f^{(2r-1)}_k(1){+}2\, \il_1^{k/2}\, f_k^{(2p+1)}(x)\,\frac{C_{2p+1}(x)}{(2p+1)!}\,dx\,.
\end{eqnarray}
All quantities on the second line of (\ref{e29}) involving $f_k$ can be expressed in terms of the $\beta_l(s)$ in (\ref{e25}--\ref{e26}). In particular, we get
\beq \label{e30}
2\,\il_1^{k/2}\,f_k(x)\,dx=k-2\,\sum_{l=0}^{\infty}\,\beta_l(s)\,\frac{k^{s-l}}{l-s+1}~,
\eq
where it has been used that for $s\in\dR$, $s<1$
\beq \label{e31}
2\,\il_0^{1/2}\,{\rm Re}\,\Bigl[\frac{1}{(1-e^{2\pi iy})^s}\Bigr]\,dy=\il_0^1\,\frac{dy}{(1-e^{2\pi iy})^s}=1
\eq
and that the same analyticity considerations as in \cite[Subsec.~2.1]{ref1} apply. In (\ref{e30}) we have to consider the cases that $s=1,2,...$ separately because of the term $l=s-1$; this will be done below.

We find, after using (\ref{e25})-(\ref{e26}) in (\ref{e29}) that for $s\neq 1,2,...$
\begin{eqnarray} \label{e32}
& \mbox{} & \hspace*{-5mm}S_k(s)=k+2\,\sum_{l=0}^{\infty}\,\beta_l(s)\,\Bigl\{\frac{1}{s-l+1}+\tfrac12+\sum_{r=1}^p\, \frac{B_{2r}}{(2r)!}\,(s-l)_{2r-1}~\nonumber \\
& & \hspace*{4cm}-~(s-l)_{2p+1}\,\il_1^{k/2}\,x^{l-s-2p-1}\,\frac{C_{2p+1}(x)}{(2p+1)!}\,dx\Bigr\}\,,
\end{eqnarray}
where $(a)_n$ is Pochhammer's symbol. The expressions in $\{\ldots\}$ at the right-hand side of (\ref{e32}) are identified in \cite[Subsec.~2.2]{ref1} as incomplete zeta functions $\zeta_{y,p}(t)$ with $y=k/2$ and $t=s-l$, for which
\beq \label{e33}
|\zeta_{y,p}(t)-\zeta(t)|=O_{t,p}(y^{-t-2p})~,~~~~~~y>0~.
\eq
Hence, for $s\neq 1,2,...$ and any $p=0,1,...$
\beq \label{e34}
S_k(s)=k+2\,\sum_{l=0}^{\infty}\,\beta_l(s)\,\zeta_{k/2,p}(s-l)\,k^{s-l}~.
\eq
In the case that $s=K=1,2,...\,$, we take the limit $s\pr K$ in (\ref{e34}), using the result, to be proved below,
\beq \label{e35}
\lim_{s\pr K}\,\frac{\beta_{K-1}(s)}{s-K}={-}\,\tfrac14~.
\eq
Note that $\beta_{K-1}(K)=0$, due to the factor $\cos(l+s)\frac{\pi}{2}$ at the right-hand side of (\ref{e26}). Hence, we get for $s=K=1,2,...$
\beq \label{e36}
S_k(K)=\tfrac12 k+2\,\sum_{l=0,l\neq K-1}^{\infty}\,\beta_l(K)\,\zeta_{k/2,p}(K-l)\,k^{K-l}~.
\eq
Finally, taking any $p=0,1,...$ with $s+2p>0$, we can use (\ref{e33}) to conclude the proof in the same way as the proof for $U_k(s)$ in (\ref{e20}) is concluded in \cite[Section~4]{ref1}.

We still have to show (\ref{e35}). From (\ref{e26}) we have for $K=1,2,...$
\beq \label{e37}
\lim_{s\pr K}\,\frac{\beta_{K-1}(s)}{s-K}=\tfrac14\,({-}1)^K\,\frac{B_{K-1}^{(K)}(0)}{(K-1)!}~,
\eq
where $B_{K-1}^{(K)}(0)/(K-1)!$ is the coefficient of $z^{K-1}$ in $(z/(e^z-1))^K$. Thus, we have for $0<r<2\pi$
\begin{eqnarray} \label{e38}
\frac{B_{K-1}^{(K)}(0)}{(K-1)!} & = & \frac{1}{2\pi i}\,\il_{|z|=r}\,\frac{1}{z^K}\,\Bigl(\frac{z}{e^z-1}\Bigr)^K\, dz~\nonumber \\
& = & \frac{1}{2\pi i}\,\il_{|z|=r}\,\Bigl(\frac{1}{e^z-1}\Bigr)^K\,dz=\frac{1}{2\pi i}\,\il_{C_r}\, \frac{1}{w^k}\:\frac{1}{1+w}\,dw=({-}1)^{K-1}~, \nonumber \\[-3mm]
\mbox{}
\end{eqnarray}
where the substitution $w=e^z-1$, $dz=dw/(1+w)$ has been used and $C_r$ is the image under the mapping $|z|=r\mapsto e^z-1$ (which is easily seen, for small $r>0$, to be a Jordan curve having the origin $w=0$ in its interior). This shows (\ref{e35}).
\end{proof}

\begin{prop} \label{note3.2}
The result in {\rm (\ref{e11})} can be obtained by taking $s=1/2$, $p=1$ in {\rm (\ref{e27})}, using that
\beq \label{e39}
\beta_0(1/2)=\frac{1}{2\sqrt{\pi}}~,~~~~~~\beta_1(1/2)=\frac{\sqrt{\pi}}{4}~,~~~~~~\beta_2(1/2)=-\frac{\pi^{3/2}}{48}~.
\eq
\end{prop}

\begin{prop} \label{note3.3}
In the lower case in {\rm(\ref{e27})} a simplification occurs since for $s=1,2,\ldots$
\beq \label{e40}
\beta_l(s)=0\,,~~k-l~{\rm odd}~;~~~~~~\zeta(s-l)=0\,,~~l=s+2,s+4,...~.
\eq
This leads to
\beq \label{e41}
S_k(s)=\tfrac12 k+\sum_{r=0}^{\left[\frac12 s\right]}\,\beta_{s-2r}(s)\,\zeta(2r)\,k^{2r}+O_{s,p} (k^{s-2p-1})~.
\eq
\end{prop}
\noindent
A similar situation as in Note \ref{note3.3} occurs in \cite[Remark 1.2]{ref1} for the case of $U_k(s)$ in (\ref{e20}) with $s=2,4,...\,$. It is concluded in \cite{ref1} that one gets exact formulas for $U_k(s)$ in that case. While this is true, see \cite{ref3}, the argument in \cite{ref1} is incomplete. In the case of $S_k(s)$ with $s=1,2,\ldots$, it can be shown that it depends polynomially on $k$ (degree $\leq s$), and so (\ref{e41}) holds exactly with the $O_{s,p}$ deleted.

\section{Bounds on $c_j$ from a representation using outer zeros} \label{sec4}
When $k$ is allowed to be unbounded, the analysis of $\mathbb{E}M_\beta(\omega;k)$ using the series in (\ref{e1}) with the $c_j$ given by (\ref{e6}) is awkward. In this section, we present an alternative representation of the $c_j$, using the zeros of $P$ in (\ref{e3}) outside $|\sigma|<1$, that is crucial for the developments in this paper.

For $j=0,1,...,k-1\,$, we let
\beq \label{e42}
\sigma_j^{\pm}=\tfrac12\,(1+\rho)\pm\sqrt{(\tfrac12\,(1+\rho))^2-\rho\,e^{2\pi ij/k}}
\eq
be the two solutions of the equation
\beq \label{e43}
\sigma(1+\rho-\sigma)=\rho\,e^{2\pi ij/k}~.
\eq
Then for $j=1,...,k-1$
\beq \label{e44}
|\sigma_j^-|<\sigma_0^-=\rho<1=\sigma_0^+<|\sigma_j^+|~,
\eq
and
\beq \label{e45}
\sigma_j^-+\sigma_j^+=1+\rho~;~~~~~~\sigma_j^-\sigma_j^+=\rho\,e^{2\pi ij/k}~.
\eq
When we let
\beq \label{e46}
\sigma_j=\sigma_j^-\,,~~j=0,1,...,k-1~;~~~~~~\sigma_{j+k}=\sigma_j^+\,,~~j=0,1,...,k-1~,
\eq
then $\sigma_j$, $j=0,1,...,2k-1$ are the $2k$ zeros of $P$ in (\ref{e3}).

\begin{lem} \label{lem4.1}
For $j=0,1,...,k-1$,
\beq \label{e47}
c_j=\frac1k\:\frac{1+\rho-\sigma_j}{(1-\sigma_j)(1+\rho-2\sigma_j)}\,\prod_{l=0}^{k-1}\,(1-\sigma_l)\,\prod_{l=0} ^{k-1}\,\Bigl(1-\frac{\sigma_l}{\sigma_j^+}\Bigr)~.
\eq
\end{lem}
\begin{proof} We have from (\ref{e6}) that
\beq \label{e48}
c_j=\frac{\sigma_j^{k-1}}{1-\sigma_j}~\frac{{\prod_{l=0}^{k-1}}\,(1-\sigma_l)} {{\prod_{l=0,l\neq j}^{k-1}} \,(\sigma_j-\sigma_l)}~.
\eq
We   re-express the product in the denominator at the right-hand side of (\ref{e48}). There holds
\beq \label{e49}
P(\sigma)=(\sigma(1+\rho-\sigma))^k-\rho^k=({-}1)^k\,\prod_{l=0}^{k-1}\, (\sigma-\sigma_l^-)\, \prod_{l=0}^{k-1}\,(\sigma-\sigma_l^+)~.
\eq
Hence, for $j=0,1,...,k-1$ from (\ref{e46})
\beq \label{e50}
P'(\sigma_j)=({-}1)^k\,\prod_{l=0,l\neq j}^{k-1}\,(\sigma_j-\sigma_l)\, \prod_{l=0}^{k-1}\,(\sigma_j-\sigma_l^+)~.
\eq
On the other hand,
\beq \label{e51}
P'(\sigma_j)=(\sigma^k(1+\rho-\sigma)^k-\rho^k)'(\sigma_j)=k\,\rho^k\,\frac{1+\rho-2\sigma_j} {\sigma_j(1+\rho-\sigma_j)}~.
\eq
Hence, for $j=0,1,...,k-1\,$,
\begin{eqnarray} \label{e52}
\prod_{l=0,l\neq j}^{k-1}\,(\sigma_j-\sigma_l) & = & \frac{({-}1)^k\,P'(\sigma_j)} {{\prod_{l=0}^{k-1}}\, (\sigma_j-\sigma_l^+)}~\nonumber \\
& = & ({-}1)^k\,k\,\rho^k\,\frac{1+\rho-2\sigma_j}{\sigma_j(1+\rho-\sigma_j)}~\frac{1}{{\prod_{l=0}^{k-1}}\, (\sigma_j-\sigma_l^+)}~.
\end{eqnarray}
Next, by the first item in (\ref{e46}), $\sigma_j-\sigma_l^+={-}(\sigma_j^+-\sigma_l)$, and so
\beq \label{e53}
\prod_{l=0,l\neq j}^{k-1}\,(\sigma_j-\sigma_l)=k\,\rho^k\,\frac{1+\rho-2\sigma_j}{\sigma_j(1+\rho-\sigma_j)}~\frac{1} {{\prod_{l=0}^{k-1}}\,(\sigma_j^+-\sigma_l)}~.
\eq
Using (\ref{e53}) in (\ref{e48}), we get
\beq \label{e54}
c_j=\frac1k~\frac{(\sigma_j/\rho)^k\,(1+\rho-\sigma_j)}{(1-\sigma_j)(1+\rho-2\sigma_j)}\,\prod_{l=0}^{k-1}\, (1-\sigma_l)\,\prod_{l=0}^{k-1}\,(\sigma_j^+-\sigma_l)~.
\eq
Finally, take out $k$ factors $\sigma_j^+$ from the last product at the right-hand side of (\ref{e54}), and use, see (\ref{e45}),
\beq \label{e55}
(\sigma_j\sigma_j^+/\rho)^k=(\sigma_j^-\sigma_j^+/\rho)^k=1
\eq
to obtain the result.\end{proof}

We   now analyze the product $\prod_{l=0}^{k-1}\,(1-\sigma_l/\sigma_j^+)$ of which $\prod_{l=0}^{k-1}\,(1-\sigma_l)$ is the special case with $j=0$.

\begin{lem} \label{lem4.2}
For $j=0,1,...,k-1$,
\beq \label{e56}
\prod_{l=0}^{k-1}\,\Bigl(1-\frac{\sigma_l}{\sigma_j^+}\Bigr)=\exp(g_j(\rho))~,
\eq
where
\beq \label{e57}
{\rm Re}(g_j(\rho))<0~,~~~~~~|g_j(\rho)|<{-}\tfrac12\,{\rm ln}\Bigl(1-\Bigl(\frac{4\rho}{(1+\rho)^2}\Bigr)^k\Bigr)~.
\eq
\end{lem}
\begin{proof}We have
\beq \label{e58}
\prod_{l=0}^{k-1}\,\Bigl(1-\frac{\sigma_l}{\sigma_j^+}\Bigr)=\exp\Bigl({-}k\,{\rm ln}\,\sigma_j^++ \sum_{l=0}^{k-1}\,{\rm ln}(\sigma_j^+-\sigma_l)\Bigr)~.
\eq
Now
\begin{eqnarray} \label{e59}
{\rm ln}(\sigma_j^+-\sigma_l) & = & {\rm ln}\,\Bigl[\sqrt{a^2-\rho\,e^{2\pi ij/k}}+\sqrt{a^2-\rho\,e^{2\pi il/k}} \,\Bigr]~\nonumber \\
& = & {\rm ln}\,\Bigl[1-a_j+\sqrt{a^2-\rho\,u_l}\,\Bigr]~,
\end{eqnarray}
where $u_l$ as defined in \eqref{e9} and
\beq \label{e60}
a=\tfrac12\,(1+\rho)~,~~~~~~1-a_j=\sqrt{a^2-\rho\ u_j}~.
\eq
Hence
\beq \label{e61}
{\rm ln}(\sigma_j^+-\sigma_l)=h_j(u_l)~,
\eq
where
\beq \label{e62}
h_j(u):={\rm ln}\,\Bigl[1-a_j+\sqrt{a^2-\rho u}\,\Bigr]~,~~~~~~|u|\leq a^2/\rho~.
\eq
Note that
\beq \label{e63}
\tau:=\frac{\rho}{a^2}=\frac{4\rho}{(1+\rho)^2}=1-\Bigl(\frac{1-\rho}{1+\rho}\Bigr)^2\in(0,1)~.
\eq
Furthermore, $1-a_j+\sqrt{a^2-\rho u}$ has positive real part when $|u|\leq a^2/\rho=\tau^{-1}$, and so $h_j(u)$ is analytic in an open set containing the closed unit disk. There is the power series representation
\beq \label{e64}
h_j(u)=\sum_{n=0}^{\infty}\,d_n(j)\,\rho^n\,u^n~,~~~~~~|u|\leq a^2/\rho~,
\eq
in which $d_n(j)$ are the power series coefficients of ${\rm ln}\,[1-a_j+\sqrt{a^2-z}]$.

From all this we get
\beq \label{e65}
\sum_{l=0}^{k-1}\,{\rm ln}(\sigma_j^+-\sigma_l)=\sum_{l=0}^{k-1}\,h_j(e^{2\pi il/k})=k\,\sum_{s=0}^{\infty}\, d_{ks}(j)\,\rho^{ks}~,
\eq
where it has been used that
\beq \label{e66}
\sum_{l=0}^{k-1}\,e^{2\pi inl/k}=\left\{\ba{lll}
k & \!\!, & ~~~n=0,k,2k,...,\\[2mm]
0 & \!\!, & ~~~{\rm otherwise}~.
\ea\right.
\eq
Noting that $d_0(j)=h_j(0)={\rm ln}\,\sigma_j^+$, we then see from (\ref{e58}) that
\beq \label{e67}
\prod_{l=0}^{k-1}\,\Bigl(1-\frac{\sigma_l}{\sigma_j^+}\Bigr)=\exp\,\Bigl(k\,\sum_{s=1}^{\infty}\,d_{ks}(j)\,\rho^{ks}\Bigr) =:\exp(g_j(\rho))~.
\eq

We shall derive an integral representation, see \eqref{e73}, for the $d_n(j)$, $n=1,2,...\,$, from which the bounds for $g_j$ readily follow. We have by Cauchy's theorem for $n=1,2,...$
\beq \label{e68}
d_n(j)=\frac{1}{2\pi i}\,\il_{|z|=r}\,\frac{{\rm ln}(1+a_j-\sqrt{a^2-z})}{z^{n+1}}\,dz
\eq
when $0<r<a^2$. We deform the integration contour $|z|=r$ so as to enclose the branch cut of $\sqrt{a^2-z}$ from $z=a^2$ to $z={+}\infty$. Now, for $x>a^2$, we have
\beq \label{e69}
\sqrt{a^2-(x\pm i0)}={\mp}\,i\,\sqrt{x-a^2}~,
\eq
and so we get for $n=1,2,...$
\beq \label{e70}
d_n(j)=\frac{1}{2\pi i}\,\il_{a^2}^{\infty}\,\Bigl[{\rm ln}(1-a_j-i\,\sqrt{x-a^2})-{\rm ln}(1-a_j+ i\sqrt{x-a^2})\Bigr]
\, \frac{dx}{x^{n+1}}~.
\eq
By partial integration, noting that the quantity in $[\ldots]$ at the right-hand side of (\ref{e70}) vanishes at $x=a^2$, we get
\begin{eqnarray} \label{e71}
d_n(j) & = & \frac{-1}{2\pi n}\,\il_{a^2}^{\infty}\,\frac{1}{2\sqrt{x-a^2}}\,\Bigl[\frac{1}{1-a_j-i\sqrt{x-a^2}} + \frac{1}{1-a_j+i\sqrt{x-a^2}}\,\Bigr]\,\frac{dx}{x^n}~\nonumber \\
& = & \frac{{-}(1-a_j)}{2\pi n}\,\il_{a^2}^{\infty}\,\frac{1}{\sqrt{x-a^2}}~\frac{1}{(1-a_j)^2+x-a^2}~\frac{dx}{x^n}
\end{eqnarray}
for $n=1,2,...\,$. Finally, setting $u^2=x-a^2\geq0$, we get
\beq \label{e72}
d_n(j)={-}\,\frac{1-a_j}{\pi n }\,\il_0^{\infty}\,\frac{1}{(1-a_j)^2+u^2}~\frac{1}{(u^2+a^2)^n}\,du~.
\eq
By the substitutions $t=u/a$ and $v=t\,b_j^{1/2}$, the result (\ref{e72}) can be brought into the forms
\begin{eqnarray} \label{e73}
d_n(j) & = & \frac{-1}{\pi\,n\,a^{2n}}\,\il_0^{\infty}\,\frac{b_j^{1/2}}{1+b_j\,t^2}~\frac{dt}{(1+t^2)^n}~\nonumber \\
& = & \frac{-1}{\pi\,n\,a^{2n}}\,\il_0^{\infty}\,\frac{1}{1+v^2}~\frac{dv}{(1+v^2/b_j)^n}~,~~~~~~n=1,2,...~,
\end{eqnarray}
where we have set $b_j=(a/(1-a_j))^2$, so that
\beq \label{e74}
\frac{1}{b_j}=1-\frac{4\rho}{(1+\rho)^2}\,e^{2\pi ij/k}
\eq
is in the right-half plane.

We   now show that ${\rm Re}\,d_n(j)<0$. We have for $t^2=x\geq0$ and $b\in\dC$, ${\rm Re}\,b>0$ that
\beq \label{e75}
{\rm Re}\,\Bigl(\frac{b^{1/2}}{1+bx}\Bigr)={\rm Re}\,\Bigl(\frac{b^{-1/2}}{b^{-1}+x}\Bigr)>0~.
\eq
Then ${\rm Re}\,d_n(j)<0$ follows from (\ref{e73}) and ${\rm Re}\,b_j^{-1}>0$. From (\ref{e67}) it is then seen that ${\rm Re}\,g_j(\rho)<0$, and the first item in (\ref{e57}) is proved.

Next, we have from ${\rm Re}\,b_j^{-1}>0$ that $|1+v^2/b_j|>1$ for all $v>0$. Hence, from (\ref{e73}),
\beq \label{e76}
|d_n(j)|<\frac{1}{2n a^{2n}}~,~~~~~~n=1,2,...~.
\eq
From (\ref{e67}) it is then seen that
\beq \label{e77}
|g_j(\rho)|<k\,\sum_{s=1}^{\infty}\,\frac{\rho^{ks}}{2ks a^{2ks}}={-}\,\tfrac12\,{\rm ln}\Bigl(1-\Bigl(\frac {\rho}{a^2}\Bigr)^k\Bigr)~,
\eq
and noting that $a=\frac12(1+\rho)$, this gives the second item in (\ref{e57}). \end{proof}

The following result follows immediately from Lemmas~\ref{lem4.1} and \ref{lem4.2}.

\begin{thm} \label{thm4.3}
For $j=0,1,...,k-1$,
\beq \label{e78}
|c_j|\leq\frac1k\,\Bigl|\frac{1+\rho-\sigma_j}{(1-\sigma_j)(1+\rho-2\sigma_j)}\Bigr|~,
\eq
and
\beq \label{e79}
c_j=\frac1k~\frac{1+\rho-\sigma_j}{(1-\sigma_j)(1+\rho-2\sigma_j)}\,\Big(1+O\Big(\frac{\tau^k}{1-\tau^k}\Big)\Big)
\eq
where $\tau=4\rho/(1+\rho)^2$ and the constant implied by $O$  in \eqref{e79} is bounded by $1$.
\end{thm}

Another inequality for $c_j$, $j=1,...,k-1\,$, is the following one.

\begin{thm} \label{thm4.4}
For $j=1,2,...,[\frac{k}{2}]=:m$,
\beq \label{e80}
|c_j|=|c_{k-j}|\leq \frac{(1-\rho)\sqrt{k}}{j}\,(1+\sqrt{2})\,c_0^{1/2}~.
\eq
\end{thm}
\begin{proof}We have from $\sigma_{k-j}=\sigma_j^{\ast}$, $j=1,...,k-1\,$, that $c_{k-j}=c_j^{\ast}$, $j=1,...,k-1\,$, and this gives $|c_j|=|c_{k-j}|$, $j=1,...,m$.

From Lemma~\ref{lem4.1} with $j=0$ we have
\beq \label{e81}
c_0=\frac1k~\,\prod_{l=1}^{k-1}\,(1-\sigma_l)^2~.
\eq
Therefore, as $\sigma_0=\rho$,
\beq \label{e82}
\frac1k\,\prod_{l=0}^{k-1}\,(1-\sigma_l)=\frac{1-\rho}{\sqrt{k}}\, \Bigl(\,\frac1k\,\prod_{l=1}^{k-1}\,(1-\sigma_l)^2 \Bigr)^{1/2}=\frac{1-\rho}{\sqrt{k}}\,c_0^{1/2}~.
\eq
Furthermore, from Lemma~\ref{lem4.2} for $j=1,2,...,k-1$
\beq \label{e83}
\Bigl|\prod_{l=0}^{k-1}\,\Bigl(1-\frac{\sigma_l}{\sigma_j^+}\Bigr)\Bigr|\leq1~,
\eq
and so from Lemma~\ref{lem4.1}
\beq \label{e84}
|c_j|\leq\Bigl|\frac{1+\rho-\sigma_j}{(1-\sigma_j)(1+\rho-2\sigma_j)}\Bigr|\,\frac{1-\rho}{\sqrt{k}}\,c_0^{1/2}~.
\eq

We shall show that for $j=1,2,...,m$
\beq \label{e85}
|1-\sigma_j|\geq\Bigl(\frac{j}{2k}\Bigr)^{1/2}\,,~~|1+\rho-2\sigma_j|\geq2\,\Bigl(\frac{j}{2k}\Bigr)^{1/2}\,,~~ |1+\rho-\sigma_j|\leq1+\sqrt{2}~,
\eq
from which the result follows at once.
We have with $t=2\pi j/k\in(0,\pi]$
\beq \label{e86}
1-\sigma_j=\tfrac12(1-\rho)+\sqrt{(\tfrac12(1+\rho))^2-\rho\,e^{it}}~,
\eq
\beq \label{e87}
1+\rho-2\sigma_j=2\,\sqrt{(\tfrac12(1+\rho))^2-\rho\,e^{it}}~,
\eq
\beq \label{e88}
1+\rho-\sigma_j=\tfrac12(1+\rho)+\sqrt{(\tfrac12(1+\rho))^2-\rho\,e^{it}}~.
\eq
Now
\beq \label{e89}
{\rm Re}\,\Bigl[\sqrt{(\tfrac12(1+\rho))^2-\rho\,e^{it}}\,\Bigr]>0~\,,~~\Bigl|\sqrt{(\tfrac12 (1+\rho))^2-\rho\,e^{it}}\Bigr|\leq\sqrt{(\tfrac12(1+\rho))^2+\rho}\:,
\eq
and so
\beq \label{e90}
|1-\sigma_j|\geq|(\tfrac12(1+\rho))^2-\rho\,e^{it}|^{1/2}\,,~~|1+\rho-2\sigma_j|\geq 2|(\tfrac12(1+\rho))^2-\rho\, e^{it}|^{1/2}
\eq
while, as $0\leq\rho\leq1$,
\beq \label{e91}
|1+\rho-\sigma_j|\leq\tfrac12(1+\rho)+\sqrt{(\tfrac12(1+\rho))^2+\rho}\leq 1+\sqrt{2}~,
\eq
which establishes the third inequality in (\ref{e85}). For the first two inequalities in (\ref{e85}), we compute
\begin{eqnarray} \label{e92}
& \mbox{} & |(\tfrac12(1+\rho))^2-\rho\,e^{it}|^2=(\tfrac12(1+\rho))^4-2\rho(\tfrac12(1+\rho))^2\cos t+\rho^2~\nonumber \\
& & =~(\tfrac12(1-\rho))^4+2\rho(\tfrac12(1+\rho))^2(1-\cos t)~\nonumber \\
& & =~(\tfrac12(1-\rho))^4+\rho(1+\rho)^2\sin^2\tfrac12\,t\geq(\tfrac12(1-\rho))^4+\rho (1+\rho)^2(t/\pi)^2~, \nonumber \\
\mbox{}
\end{eqnarray}
where the inequality $\sin x\geq 2x/\pi$, $0\leq x\leq\pi/2$ has been used. Now for $0<y\leq1$ and $0\leq\rho\leq1$
\beq \label{e93}
\frac{(\tfrac12(1-\rho))^4}{y^2}+\rho(1+\rho)^2\geq(\tfrac12(1-\rho))^4+\rho(1+\rho)^2
= ((\tfrac12(1+\rho))^2+\rho)^2\geq(\tfrac12)^4~,
\eq
and so we get
\beq \label{e94}
|(\tfrac12(1+\rho))^2-\rho\,e^{it}|^2\geq(\tfrac12)^4\,\Bigl(\frac{t}{\pi}\Bigr)^2~,~~~~~~0\leq t\leq\pi~.
\eq
This yields the first two inequalities in (\ref{e85}).\end{proof}

\section{Extension of Theorem~\ref{thm2.1}} \label{sec5}
In this section we   show the following extension of Theorem~\ref{thm2.1}.

\begin{thm} \label{thm5.1}
For fixed $\beta>0$,
\beq \label{e95}
\mathbb{E}M_\beta(\omega;k)=\frac{1}{2\beta}-\Bigl(\frac{1}{2k\omega}\Bigr)^{1/2}\Bigl(k-\sum_{j=1}^{k-1}\, \frac{1}{\sqrt{1-u_j}}\Bigr)+O(\omega^{-1})\,,~~~~\omega\pr\infty~,
\eq
where $O$ holds uniformly in $k=1,2,...\,$.
\end{thm}
\begin{proof}We first show that we can restrict attention to the term $j=0$ in the series (\ref{e1}) for $\mathbb{E}M_\beta(\omega;k)$. We have by (\ref{e4})
\beq \label{e96}
\Bigl|\sum_{j=1}^{k-1}\,\frac{c_j\sigma_j}{\gamma_2(1-\sigma_j)}\Bigr|= \Bigl|\frac{1-\rho}{2\beta}\, \sum_{j=1}^{k-1}\,\frac{c_j\sigma_j}{1-\sigma_j}\Bigr|\leq\frac{1-\rho}{\beta}\,\sum_{j=1}^m\,\Bigl| \frac{c_j\sigma_j}{1-\sigma_j}\Bigr|~,
\eq
where $m=[\frac12 k]$ and where it has been used that $c_{k-j}=c_j^{\ast}$, $\sigma_{k-j}=\sigma_j^{\ast}$, $j=1,2,...,k-1$. Now Theorem~\ref{thm4.4} and (\ref{e85}) give
\begin{eqnarray} \label{e97}
& \mbox{} & \frac{1-\rho}{\beta}\,\sum_{j=1}^m\,\Bigl|\frac{c_j\sigma_j}{1-\sigma_j}\Bigr|\leq \frac{1-\rho}{\beta}\,(1+\sqrt{2})\,c_0^{1/2}\,\sum_{j=1}^m\,(1-\rho)\,\frac{\sqrt{k}}{j}\, \Bigl(\frac{2k}{j}\Bigr)^{1/2}~\nonumber \\
& & \leq~\frac{\sqrt{2}+2}{\beta}\,c_0^{1/2}\,\zeta(\tfrac32)(1-\rho)^2 k=(\sqrt{2}+2)c_0^{1/2}\, \zeta(\tfrac32)\,\frac{2\beta\rho}{\omega}~,
\end{eqnarray}
where (\ref{e5}) has been used in the last step. Hence,
\beq \label{e98}
\mathbb{E}M_\beta(\omega;k)=\frac{c_0\rho}{2\beta}+c_0^{1/2}\,\beta\,O(\omega^{-1})~,
\eq
where the constant implied by $O$ is bounded by $2(\sqrt{2}+2)\zeta(\frac32)$.

We   now bound and approximate $c_0$ with \eqref{e107}, \eqref{e113} as a result. We have from Lemma~\ref{lem4.1} with $j=0$
\beq \label{e99}
c_0=\frac1k\,\prod_{j=1}^{k-1}\,(1-\sigma_j)^2~.
\eq
Furthermore, there is the approximation, see the appendix,
\beq \label{e100}
1-\sigma_j=\sqrt{1-u_j}\,\Bigl(1-\eps\Bigl(1-\frac{1}{\sqrt{1-u_j}}\Bigr)+O\Bigl(\frac{(1-\rho)^3}{\sqrt{1-u_j}} \Bigr)\Bigr)
\eq
with $\eps=\frac12(1-\rho)+\frac18(1-\rho)^2$, and
\beq \label{e101}
|\sqrt{1-u_j}|=|\sqrt{1-u_{k-j}}|\geq 2(j/k)^{1/2}~,~~~~~~j=1,2,...,m~.
\eq
We furthermore have from \eqref{e9} that
\beq \label{e102}
\frac1k\,\prod_{j=1}^{k-1}\,(1-u_j)=\frac{2^{k-1}}{k}\,\prod_{j=1}^{k-1}\,\sin\frac{\pi j}{k}=1~,
\eq
see \cite{ref1}, (1.11) for the last identity. From all this we get
\beq \label{e103}
c_0=\prod_{j=1}^{k-1}\,\Bigl(1-\eps\Bigl(1-\frac{1}{\sqrt{1-u_j}}\Bigr)+O\Bigl(\frac{(1-\rho)^3}{\sqrt{1-u_j}} \Bigr)\Bigr)~.
\eq

Next, from (\ref{e101}) and $\eps=O(1-\rho)$, we see that
\beq \label{e104}
\eps\Bigl(1-\frac{1}{\sqrt{1-u_j}}\Bigr)=O\Bigl(\frac{1}{\sqrt{\omega}}\Bigr)
\eq
uniformly in $j=1,2,...,k-1\,$. Hence
\beq \label{e105}
c_0=\prod_{j=1}^{k-1}\,\Bigl(1-\eps\Bigl(1-\frac{1}{\sqrt{1-u_j}}\Bigr)\Bigr)^2\cdot \Bigl(1+O\Bigl(\sum_{j=1}^m\, \Bigl|\frac{(1-\rho)^3}{\sqrt{1-u_j}}\Bigr|\Bigr)\Bigr)~.
\eq
Now by (\ref{e101}) and $m=[\frac12 k]$
\begin{eqnarray} \label{e106}
& \mbox{} & (1-\rho)^3\,\sum_{j=1}^m\,\Bigl|\frac{1}{\sqrt{1-u_j}}\Bigr|\leq\tfrac12\,k^{1/2}(1-\rho)^3 \,\sum_{j=1}^m\,\frac{1}{j^{1/2}}~\nonumber \\
& & \leq~\tfrac12\,k^{1/2}(1-\rho)^3\,2(\tfrac12\,k)^{1/2}=\tfrac12\,k(1-\rho)^3\,\sqrt{2}= O\Bigl(\frac{1}{\omega^{3/2}k^{1/2}}\Bigr)\,.
\end{eqnarray}
Hence
\beq \label{e107}
c_0=\hat{c}_0\Bigl(1+O\Bigl(\frac{1}{\omega^{3/2}k^{1/2}}\Bigr)\Bigr)
\eq
with
\begin{eqnarray} \label{e108}
\hat{c}_0 & = & \prod_{j=1}^{k-1}\,\Bigl(1-\eps\Bigl(1-\frac{1}{\sqrt{1-u_j}}\Bigr)\Bigr)^2~\nonumber \\
& = & \exp\Bigl(2\,\sum_{j=1}^{k-1}\,{\rm ln}\Bigl(1-\eps\Bigl(1-\frac{1}{\sqrt{1-u_j}}\Bigr)\Bigr)\Bigr)~.
\end{eqnarray}
We develop
\begin{eqnarray} \label{e109}
& \mbox{} & \sum_{j=1}^{k-1}\,{\rm ln}\Bigl(1-\eps\Bigl(1-\frac{1}{\sqrt{1-u_j}}\Bigr)\Bigr)={-}\eps\,\sum_{j=1}^{k-1}\,\Bigl(1-\frac{1}{\sqrt{1-u_j}}\Bigr)\nonumber \\
& & ~{-}\tfrac12 \eps^2\,\sum_{j=1}^{k-1}\,\Bigl(1-\frac{1}{\sqrt{1-u_j}}\Bigr)^2+O\,\Bigl[\eps^3\,\sum_{j=1}^{k-1}\, \Bigl|1-\frac{1}{\sqrt{1-u_j}}\Bigr|^3\Bigr]~.
\end{eqnarray}
From
\beq \label{e110}
\sum_{j=1}^{k-1}\,\Bigl(1-\frac{1}{\sqrt{1-u_j}}\Bigr)=O(\sqrt{k})~, \quad \sum_{j=1}^{k-1}\,\frac{1}{1-u_j}=\tfrac12(k-1)
\eq
(see \eqref{e11} for the first item in \eqref{e110} and use
 Theorem~\ref{thm3.1} or proceed directly for the second item in \eqref{e110})
we have
\beq \label{e112}
\sum_{j=1}^{k-1}\,\Bigl(1-\frac{1}{\sqrt{1-u_j}}\Bigr)^2=O(k)~.
\eq
Finally, from \eqref{e101} we have
\begin{eqnarray} \label{e1111}
\sum_{j=1}^{k-1}\, \Bigl|1-\frac{1}{\sqrt{1-u_j}}\Bigr|^3=O\Bigl(k^{3/2}\sum_{j=1}^m j^{-3/2}\Bigr)=O(k^{3/2}).
\end{eqnarray}
Using (\ref{e109}--\ref{e1111}) in (\ref{e108}), we get
\begin{eqnarray} \label{e113}
\hat{c}_0 & = & 1-2\eps\,\sum_{j=1}^{k-1}\,\Bigl(1-\frac{1}{\sqrt{1-u_j}}\Bigr)+O(\eps^2\,k)+O(\eps^3\,k^{3/2})~\nonumber \\
& = & 1-(1-\rho)\,\sum_{j=1}^{k-1}\,\Bigl(1-\frac{1}{\sqrt{1-u_j}}\Bigr)+O(\omega^{-1})+O(\omega^{-3/2})~,
\end{eqnarray}
where we have used that $2\eps=1-\rho+O((1-\rho)^2)$ and that $k(1-\rho)^2=O(\omega^{-1})$, see (\ref{e5}).

Using (\ref{e113}) and (\ref{e107}) in (\ref{e98}), we get
\begin{eqnarray} \label{e114}
& \mbox{} & \mathbb{E}M_\beta(\omega;k)=\frac{\hat{c}_0\rho}{2\beta}+O(\omega^{-1})~\nonumber \\
& & =~\frac{\rho}{2\beta}\,\Bigl(1-(1-\rho)\,\sum_{j=1}^{k-1}\,\Bigl(1-\frac{1}{\sqrt{1-u_j}}\Bigr)\Bigr)+
O(\omega^{-1})~\nonumber \\
& & =~\frac{1}{2\beta}-\frac{1-\rho}{2\beta}\,\Bigl(1+\rho\,\sum_{j=1}^{k-1}\,\Bigl(1-\frac{1} {\sqrt{1-u_j}}\Bigr)\Bigr)+O(\omega^{-1})~\nonumber \\
& & =~\frac{1}{2\beta}-\frac{1-\rho}{2\beta}\,\Bigl(k-\sum_{j=1}^{k-1}\,\frac{1}{\sqrt{1-u_j}}\Bigr)+
O(\omega^{-1})
\end{eqnarray}
where (\ref{e110}) has been used in the last step to replace the $\rho$ in front of $\sum_j$ by 1 at the expense of an error $O((1-\rho)^2\, \sqrt{k})=O(\omega^{-1})$. Since by (\ref{e5})
\beq \label{e115}
\frac{1-\rho}{2\beta}=\frac{\rho^{1/2}}{\sqrt{2k\omega}}=\frac{1}{\sqrt{2k\omega}}+O\Bigl(\frac{1}{k\omega} \Bigr)~,
\eq
we get the result.\end{proof}

\section{Behavior of $\mathbb{E}M_\beta(\omega;k) $ as $\omega\pr0$ and $k\pr\infty$} \label{sec6}
In this section we show the following result.

\begin{thm} \label{thm6.1}
Assume that $\beta>0$ is fixed and that $\omega\pr0$ and $k\pr\infty$. Then $\mathbb{E}M_\beta(\omega;k) \pr0$.
\end{thm}
\begin{proof}By Theorem~\ref{thm4.3}
\beq \label{e116}
c_j=\frac1k~\frac{1+\rho-\sigma_j}{(1-\sigma_j)(1+\rho-2\sigma_j)}\,(1+\eps_{j,k}(\rho))~,
\eq
where for all $k=1,2,...$ and all $j=0,1,...,k-1$
\beq \label{e117}
|\eps_{j,k}(\rho)|\leq \frac{\tau^k}{1-\tau^k}~;~~~~~~\tau=\frac{4\rho}{(1+\rho)^2}=1-\Bigl(\frac{1-\rho}{1+\rho}\Bigr)^2~.
\eq
We have by (\ref{e5})
\beq \label{e118}
k(1-\rho)^2=\frac{2\rho\beta^2}{\omega}\pr\infty~,
\eq
and so
\beq \label{e119}
\tau^k\leq\exp\Bigl(\frac{{-}k(1-\rho)^2}{(1+\rho)^2}\Bigr)\pr0~.
\eq

Now by (\ref{e1}), (\ref{e4}) and (\ref{e116})
\begin{eqnarray} \label{e120}
\mathbb{E}M_\beta(\omega;k) & = & \frac{1-\rho}{2\beta}\,\sum_{j=0}^{k-1}\,\frac1k~ \frac{(1+\rho-\sigma_j)\, \sigma_j}{(1-\sigma_j)^2(1+\rho-2\sigma_j)}~\nonumber \\
& & +~\frac{\rho\,\eps_{0,k}(\rho)}{2\beta k(1-\rho)^2}+\frac{1-\rho}{2\beta k}\, \sum_{j=1}^{k-1}\, \frac{(1+\rho-\sigma_j)\,\sigma_j\,\eps_{j,k}(\rho)}{(1-\sigma_j)^2(1+\rho-2\sigma_j)}~.
\end{eqnarray}
The second term on the right-hand side of (\ref{e120}) tends to 0 by (\ref{e117})-(\ref{e119}). As to the third term on the right-hand side of (\ref{e120}), we estimate
\begin{eqnarray} \label{e121}
& \mbox{} & \Bigl|\frac{1-\rho}{2\beta k}\,\sum_{j=1}^{k-1}\,\frac{(1+\rho-\sigma_j)\,\sigma_j\,\eps_{j,k}(\rho)} {(1-\sigma_j)^2(1+\rho-2\sigma_j)}\Bigr|~\nonumber \\
& & \leq~\frac{(1-\rho)\tau^k}{2\beta k(1-\tau^k)}\,\sum_{j=1}^{k-1}\,\Bigl|\frac{(1+\rho-\sigma_j)\,\sigma_j} {(1-\sigma_j)^2(1+\rho-2\sigma_j)}\Bigr|~\nonumber \\
& & \leq~\frac{(1-\rho)\tau^k}{\beta k (1-\tau^k)}\,\sum_{j=1}^m\,\frac{(1+\sqrt{2})\cdot1}{\dfrac{j}{2k}\cdot2\Bigl(\dfrac{j}{2k}\Bigr)^{1/2}}\leq \frac{\sqrt{2}+2}{\beta(1-\tau^k)}\,\zeta(\tfrac32)\,\tau^k(1-\rho)\,\sqrt{k}~.
\end{eqnarray}

Here we have used (\ref{e85}), with $m=[\frac12 k]$, and $|\sigma_j|\leq1$. By (\ref{e118}) and (\ref{e119}) we have that $\tau^k(1-\rho)\,\sqrt{k}\pr0$, and so also the third term at the right-hand side of (\ref{e120}) tends to 0.

We finally consider the first term,
\beq \label{e122}
R_k:=\frac{1-\rho}{2\beta k}\,\sum_{j=0}^{k-1}\,\frac{(1+\rho-\sigma_j)\,\sigma_j} {(1-\sigma_j)^2(1+\rho-2\sigma_j)}
\eq
on the right-hand side of (\ref{e120}). We   show below that
\beq \label{e123}
0\leq R_k\leq\frac{1+\rho}{2\beta(1-\rho)}~\frac{1}{\sqrt{\pi(k-3/4)}}~\frac{\tau^k}{1-\tau^k}~.
\eq
From (\ref{e118}) and (\ref{e119}) it then follows that also $R_k\pr0$.

To show (\ref{e123}), we follow the approach that was used to prove Lemma~\ref{lem4.2}, and we let for $|u|<\tau^{-1}$
\beq \label{e124}
F(u)=\frac{(1+\rho-\sigma(u))\,\sigma(u)}{(1-\sigma(u))^2(1+\rho-2\sigma(u))}~,
\eq
where
\beq \label{e125}
\sigma(u)=a-\sqrt{a^2-z}~;~~~~~~a=\tfrac12(1+\rho)\,,~~z=\rho u~.
\eq
Using
\beq \label{e126}
(1+\rho-\sigma(u))\,\sigma(u)=\rho u~,
\eq
we have
\begin{eqnarray} \label{e127}
F(u)=\frac{\rho u}{(1-\sigma(u))^2(1+\rho-2\sigma(u))}& = & \frac{z}{2(1-a+\sqrt{a^2-z})^2\,\sqrt{a^2-z}}~\nonumber \\
& = & \sum_{n=0}^{\infty}\,g_n\,z^n~.
\end{eqnarray}
By contour integration as in (\ref{e68})--(\ref{e71}), we have that
\begin{eqnarray} \label{e128}
g_n & = & \frac{1}{2\pi i}\,\il_{|z|=r}\,\frac{z}{2(1-a+\sqrt{a^2-z})^2\,\sqrt{a^2-z}}~\frac{dz}{z^{n+1}}~\nonumber \\
& = & \frac{1}{2\pi}\,\il_{a^2}^{\infty}\,\frac{(1-a)^2+a^2-x}{((1-a)^2-a^2+x)^2}~\frac{1}{\sqrt{x-a^2}}~\frac{dx}{x^n}~\nonumber \\
& = & \frac{1}{2\pi}\,\il_{a^2}^{\infty}\,\frac{\frac12(1+\rho^2)-x}{(x-\rho)^2}~\frac{1}{\sqrt{x-a^2}}~\frac{dx}{x^n}
\end{eqnarray}
for $n=0,1,...\,$. Now $g_0=0$, see (\ref{e127}), and so the last integral vanishes for $n=0$. The integrand in this integral changes sign once, from positive to negative at $x=\frac12(1+\rho^2)>a^2$, and $1/x^n$ is positive and strictly decreasing in $x\geq a^2$ when $n=1,2,...\,$. It follows that $g_n>0$, $n=1,2,...\,$. Also, we have
\beq \label{e129}
\frac{\frac12(1+\rho^2)-x}{(x-\rho)^2}\leq\frac{\frac12(1+\rho^2)-a^2}{(a^2-\rho)^2}=\frac{4}{(1-\rho)^2}~,~~~~~~ x\geq a^2~,
\eq
and so we conclude that for $n=1,2,...$
\begin{eqnarray} \label{e130}
0~<~g_n & \leq & \frac{1}{2\pi}~\frac{4}{(1-\rho)^2}\,\il_{a^2}^{\infty}\,\frac{1}{\sqrt{x-a^2}}~\frac{dx}{x^n}~\nonumber \\
& = & \frac{1}{\pi}~\frac{1+\rho}{(1-\rho)^2}~\frac{1}{a^{2n}}\,\il_1^{\infty}\,\frac{1}{\sqrt{t-1}}~\frac{dt}{t^n}~.
\end{eqnarray}
From $\frac{1}{n!}\,F^{(n)}(0)=g_n\,\rho^n$ and $\tau=\rho/a^2$, we then get
\beq \label{e131}
0<\frac{F^{(n)}(0)}{n!}\leq\frac{1}{\pi}~\frac{1+\rho}{(1-\rho)^2}\,\il_1^{\infty}\,\frac{(\tau/t)^n}{\sqrt{t-1}}\,dt ~,~~~~~~n=1,2,...~.
\eq

We return to (\ref{e122}). As in (\ref{e65}), we have
\beq \label{e132}
R_k=\frac{1-\rho}{2\beta}~\frac1k\,\sum_{j=0}^{k-1}\,F(e^{2\pi ij/k})=\frac{1-\rho}{2\beta}\,\sum_{s=0}^{\infty}\, \frac{F^{(ks)}(0)}{(ks)!}~.
\eq
Since $F(0)=0$, we obtain from (\ref{e131}) that
\begin{eqnarray} \label{e133}
0~<~R_k & \leq & \frac{1-\rho}{2\beta}~\frac{1}{\pi}~\frac{1+\rho}{(1-\rho)^2}\,\sum_{s=1}^{\infty}\,\il_1^{\infty}\, \frac{(\tau/t)^{ks}}{\sqrt{t-1}}\,dt~\nonumber \\
& = & \frac{1}{2\pi\beta}~\frac{1+\rho}{1-\rho}\,\il_1^{\infty}\,\frac{(\tau/t)^k}{1-(\tau/t)^k}~\frac{dt}{\sqrt{t-1}}~.
\end{eqnarray}
Then using that
\beq \label{e134}
0<\frac{(\tau/t)^k}{1-(\tau/t)^k}<\frac{\tau^k}{1-\tau^k}~\frac{1}{t^k}~,~~~~~~t>1~,
\eq
we obtain
\beq \label{e135}
0<R_k<\frac{1}{2\pi\beta}~\frac{1+\rho}{1-\rho}~\frac{\tau^k}{1-\tau^k}\,\il_1^{\infty}\,\frac{dt} {t^k\,\sqrt{t-1}}~.
\eq
For the remaining integral, we use the substitution $t=e^s$, $s\geq0$ and the inequality $e^{s/2}-e^{-s/2}>s$, $s>0$, and we get
\beq \label{e136}
\il_1^{\infty}\,\frac{dt}{t^k\sqrt{t-1}}=\il_0^{\infty}\,\frac{e^{-(k-3/4)s}}{\sqrt{e^{s/2}-e^{-s/2}}}\,ds< \il_0^{\infty}\,s^{-1/2}\,e^{-(k-3/4)s}\,ds~.
\eq
The last integral in (\ref{e136}) equals $(\pi/(k-3/4))^{1/2}$, and using this in (\ref{e135}) we get (\ref{e123}). The proof is complete. \end{proof}

\begin{prop} \label{note6.2}
From the estimates of the three terms at the right-hand side of {\rm(\ref{e130})}, it is seen that
\beq \label{e137}
\mathbb{E}M_\beta(\omega;k)=O(\tau^k({-}{\rm ln}\,\tau^k)^{1/2})~.
\eq
\end{prop}

\appendix
\section{Approximating $1-\sigma_j$}
We present approximations of $1-\sigma_j$, $j=1,2,...,k-1\,$, that were needed at several places when $(1-\rho)\,\sqrt{k}$ is small. With $u_j$ as in \eqref{e9}, we have
\begin{eqnarray} \label{e138}
1-\sigma_j & = & \tfrac12(1-\rho)+\sqrt{(\tfrac12(1+\rho))^2-\rho\,u_j}~\nonumber \\
& = & \tfrac12(1-\rho)+\sqrt{1-u_j-(1-\rho)(1-u_j)+\tfrac14(1-\rho)^2}~\nonumber \\
& = & \tfrac12(1-\rho)+\sqrt{1-u_j}\,\Bigl(1-(1-\rho)+\frac{(1-\rho)^2}{4\sqrt{1-u_j}}\Bigr)^{1/2}~.
\end{eqnarray}
We have for $j=1,2,...,[\frac12 k]$
\beq \label{e139}
|1-u_{k-j}|=|1-u_j|=2\sin\Bigl(\frac{\pi j}{k}\Bigr)\geq\frac{4j}{k}~,
\eq
and so
\beq \label{e140}
\Bigl|\frac{(1-\rho)^2}{4\sqrt{1-u_j}}\Bigr|\leq\tfrac18(1-\rho)^2\,k^{1/2}~,~~~~~~j=1,2,...,k-1~.
\eq
We develop the square root on the last line in (\ref{e138}) under the condition that
\beq \label{e141}
1-\rho<\tfrac13~,~~~~~~\tfrac18(1-\rho)^2\,k^{1/2}<\tfrac13~.
\eq
Then we get
\begin{eqnarray} \label{e142}
1-\sigma_j & = & \tfrac12(1-\rho)+\sqrt{1-u_j}\,\Bigl(1-\tfrac12(1-\rho)+\frac{(1-\rho)^2}{8 \sqrt{1-u_j}}~\nonumber \\
& & -~\tfrac18\Bigl({-}(1-\rho)+\frac{(1-\rho)^2}{4\sqrt{1-u_j}}\Bigr)^2+...\Bigr)~\nonumber \\
& = & \tfrac12(1-\rho)+\sqrt{1-u_j}-\tfrac12(1-\rho)\,\sqrt{1-u_j}+\tfrac18(1-\rho)^2~\nonumber \\
& & -~\tfrac18(1-\rho)^2\,\sqrt{1-u_j}+\frac{1}{16}(1-\rho)^3-\frac{1}{128}~\frac{(1-\rho)^4}{\sqrt{1-u_j}}+...~\nonumber \\
& = & \sqrt{1-u_j}\,\Bigl(1-\tfrac12\Bigl(1-\frac{1}{\sqrt{1-u_j}}\Bigr)(1-\rho+\tfrac14(1-\rho)^2)~\nonumber \\
& & +~\frac{(1-\rho)^3}{16\sqrt{1-u_j}}-\frac{1}{128}~\frac{(1-\rho)^4}{1-u_j}+...\Bigr)~\nonumber \\
& = & \sqrt{1-u_j}\,\Bigl(1-\tfrac12\Bigl(1-\frac{1}{\sqrt{1-u_j}}\Bigr)(1-\rho+\tfrac14(1-\rho)^2)+ O \Bigl(\frac{(1-\rho)^3}{\sqrt{1-u_j}}\Bigr)\Bigr)~.
\end{eqnarray}


\begin{thebibliography}{99}
\bibitem{adanzhao} Adan. I.J.B.F., Y. Zhao (1996). Analyzing $GI/E_r/1$ queues. \textit{Oper. Res. Letters} \textbf{19}: 183-190.

\bibitem{ref111} Albrecher, H., E.C.K. Cheung, S. Thonhauser (2011). Randomized observation times for the compound Poisson risk model: The discounted penalty function. To appear in \textit{Scandinavian Actuarial Journal}.

\bibitem{asmussen} Asmussen, S. (2003). \textit{Applied Probability and Queues} (second edition),
Springer-Verlag, New York.

\bibitem{asmussenglynnpitman} Asmussen, S., P. Glynn, J. Pitman (1995). Discretization error in simulation of one-dimensional reflecting Brownian motion. \textit{Ann. Appl. Probab.} \textbf{5}: 875-896.

\bibitem{ref1}
Brauchart, J.S., D.P.\ Hardin, E.B.\ Saff (2009). The Riesz energy of the N-th roots of unity: an asymptotic expansion for large N.
\textit{Bull. London Math. Soc.} {\bf 41}: 621-633.

\bibitem{debruijn} De Bruijn, N.G. (1981). \textit{Asymptotic Methods in Analysis},
Dover Publications, New York.

\bibitem{calvin} Calvin, J. (1995). Average performance of nonadaptive algorithms for global optimization.
\textit{Journal of Mathematical Analysis and Applications} \textbf{191}: 608-617.

\bibitem{changperes} Chang, J.T., Y. Peres (1997). Ladder heights, Gaussian random walks and the Riemann zeta function. \textit{Ann. Probab.} \textbf{25}: 787-802.

\bibitem{chenyao} Chen, H., D.D. Yao (2001). \textit{Fundamentals of Queueing Networks},
Springer-Verlag, New York.






\bibitem{ref3}
Fisher, M.E. (1971). Solutions to Problem 69-14, ``Sum of Inverse Powers of Cosines''. {\it SIAM Review} {\bf 13}: 116--119.

\bibitem{jllerch}
Janssen, A.J.E.M., J.S.H. van Leeuwaarden (2006). On Lerch's transcendent and the Gaussian random walk. \textit{Ann. Appl. Probab.} {\bf 17}: 421-439.

\bibitem{ref2} Kuznetsov, A., A. Kyprianou, J.C. Pardo, K. van Schaik (2011). A Wiener-Hopf Monte Carlo simulation technique for L\'{e}vy processes. \textit{Ann. Appl. Probab.} {\bf 21}: 2171-2190.


\bibitem{lev} Williams, D. (1991). {\it Probability with Martingales}, Cambridge University Press, Cambridge.

\end{thebibliography}
\end{document}